\documentclass[a4paper,12pt]{amsart}

\usepackage{amssymb,mathrsfs,bm}
\usepackage{mathtools,galois,braket,esint}
\usepackage[utf8]{inputenc}
\usepackage[margin=0.5in]{geometry}
\usepackage{hyperref,hypcap}
\hypersetup{colorlinks=true,allcolors=black,bookmarksdepth=3}
\usepackage[inline]{enumitem}
\usepackage{array}
\usepackage[utf8]{inputenc}
\usepackage{tikz, tikz-3dplot}
\usepackage{amsmath, amsthm, amssymb,graphics }
\usetikzlibrary{arrows,shapes,positioning}
\usetikzlibrary{intersections}
\usetikzlibrary{decorations.pathreplacing,calligraphy}

\usepackage{amsmath,amsthm,amssymb,amscd,color, xcolor,mathtools,url,tikz}
\usepackage{bbm}
\usepackage{pgfplots}

\theoremstyle{plain}
\newtheorem{theorem}{Theorem}[section]
\newtheorem{lemma}[theorem]{Lemma}
\newtheorem{proposition}[theorem]{Proposition}

\theoremstyle{definition}
\newtheorem{definition}[theorem]{Definition}

\theoremstyle{remark}

\numberwithin{equation}{section}
\numberwithin{figure}{section}
\numberwithin{table}{section}
\newcommand{\dd}{\mathop{}\!\mathrm{d}}

\allowdisplaybreaks

\newcommand{\comment}[1]{\vskip.1cm
\fbox{%
\parbox{0.93\linewidth}{\footnotesize #1}}
\vskip.1cm
}

\begin{document}

\author[A. Constantin]{Adrian Constantin}
\address{Adrian Constantin, Faculty of Mathematics, University of Vienna, Oskar-Morgenstern-Platz 1, 1090 Vienna, Austria}
\email{adrian.constantin@univie.ac.at}

\author[D. Dritschel]{David Dritschel}
\address{David Dritschel, Mathematical Institute, University of St Andrews, St Andrews KY16 9SS, UK}
\email{david.dritschel@st-andrews.ac.uk}

\author[P. Germain]{Pierre Germain}
\address{Pierre Germain, Department of Mathematics, Huxley Building, South Kensington Campus,
Imperial College London, London SW7 2AZ, United Kingdom}
\email{pgermain@ic.ac.uk}

\title{On the onset of filamentation on two-dimensional vorticity
interfaces}

\begin{abstract} We study an asymptotic nonlinear model for filamention on two-dimensional vorticity interfaces. Different re-formulations of the model equation reveal its underlying structural properties. They enable us to construct global weak solutions and to prove the existence of traveling waves.
\end{abstract}

\maketitle

\setcounter{tocdepth}{1}

\tableofcontents

\section{Introduction}

\subsection{The equation}
We consider complex functions $u$ on the torus $\mathbb{T} = \mathbb{R} / (2\pi \mathbb{Z})$ with positive spectrum $\mathbb{P}_+ u = u$, where $\mathbb{P}_+$ is the projector defined by $\mathbb{P}_+ e^{ikx} = \mathbf{1}_{[1,\infty)} (k) e^{ikx}$. We add a time dependence $u=u(t,x)$ and evolve such functions by the following equation
\begin{equation}
\label{eqC} \tag{$\star$}
\partial_t u = \partial_x \mathcal{C}_\sigma[u]
\end{equation}
where the cubic term is given by
$$
\mathcal{C}_\sigma [u](x) = \mathbb{P}_+ \left[ \frac{1}{4\pi} \int_0^{2\pi} \frac{|u(x) - u(y)|^2 (u(x) - u(y))}{1-\cos (x-y)} \dd y - \sigma |u(x)|^2 u(x) \right],
$$
for a coefficient $\sigma$ equal to $0$ or $1$. Equation \eqref{eqC} is referred to as the `filamentation equation'.
\begin{itemize}
\item $\sigma=0$ corresponds to the periodic planar case (a linear vorticity interface in the plane)
\item $\sigma=1$ corresponds to the spherical case (a zonal vorticity interface on the sphere)
\end{itemize}

\subsection{Known results} The filamentation equation \eqref{eqC} was first derived by Dritschel \cite{Dritschel} in a different form, as an asymptotic model describing the onset of filamentation on two-dimensional vorticity interfaces. To be more specific, this equation arose in \cite{Dritschel} as the normal form (valid for small wave-slope solutions) of a weakly nonlinear expansion of the two-dimensional Euler equation, specifically the \textit{contour dynamics} equation applicable to piecewise-uniform vorticity\cite{Zabusky1979,Dritschel1988o}.

The filamentation equation for $\sigma=0$ can also be seen as the normal form to cubic order of the Burgers-Hilbert equation, see Biello and Hunter \cite{BH}. These authors introduced the Burgers-Hilbert equation as a universal model for constant-dispersion Hamiltonian systems, including in particular two-dimensional vorticity interfaces\footnote{Though it appears that, in the regime of small solutions, the correct equation is the one derived in \cite{Dritschel}, since cubic terms cannot be neglected.}. The Burgers-Hilbert equation was first written down by Marsden and Weinstein \cite{MW}; it has stimulated active research over the last decade, see \cite{BN,CCZ,DGS,HITW,Yang} and the review \cite{Hunter}.

In the present article, we revisit the properties and analysis of the filamentation equation. The key observation is that the nonlinearity can be written in the simple form appearing in \eqref{eqC}. Combining this observation with its spectral formulation, we are able to uncover some interesting properties. In the companion paper \cite{DCG}, the derivation of \eqref{eqC} is revisited and its finite-time blow up dynamics are studied.

\subsection{Plan of the article and results obtained} In Section \ref{sectionstructure}, we review the algebraic structure of the equation (conservation laws, symmetries) and give its spectral (Fourier space) expression, relying on computations appearing in appendices \ref{appendixA} and \ref{appendixB}.
In Section \ref{sectionstrong}, we sketch the construction of strong solutions, slightly improving over the result mentioned in \cite{Hunter}.
In Section \ref{sectionweak}, weak solutions, which were conjectured to exist in \cite{Hunter}, are constructed. The proof exploits the coercivity of the Hamiltonian and an appropriate weak formulation.
Finally, Section \ref{sectionvariational} examines variational problems involving the conserved quantities of the equation, resulting in the construction of new traveling wave solutions in Section \ref{sectiontraveling} --- new explicit formulas for traveling waves are also exhibited there for the case $\sigma=1$. Once again, the coercivity of the Hamiltonian is key in the proof.

\subsection{Notations}
A general function $u$ can be expanded in a Fourier series as
$$
u(x) = \sum_{k=-\infty}^\infty a_k e^{ikx}, \qquad a_k = \frac{1}{2\pi} \int_0^{2\pi} u(x) e^{-ikx} \dd x.
$$

Throughout this article, solutions of the equation \eqref{eqC} have positive frequencies only, hence their Fourier expansion reduces to 
$$
u(x) = \sum_{k=1}^\infty a_k e^{ikx}, \qquad a_k = \frac{1}{2\pi} \int_0^{2\pi} u(x) e^{-ikx} \dd x.
$$

Different Fourier multipliers appear in this article; recall that a Fourier multiplier $M$ is defined through its symbol $\mu(k)$ by the relation
$$
M \left[\sum_{k=-\infty}^\infty a_k e^{ikx} \right] = \sum_{k=-\infty}^\infty \mu(k) a_k e^{ikx}.
$$
The table below recapitulates different Fourier multipliers which will be used.

\medskip

\begin{center}
\begin{tabular}{ |c|c|c| } 
 \hline
Operator & Notation  & Symbol \\ 
  \hline
Projector on non-negative frequencies (Szeg\"o) & $\mathbb{P}$ & $\mathbf{1}_{[0,\infty)(k)}$ \\ 
 \hline
Projector on positive frequencies & $\mathbb{P}_+$ & $\mathbf{1}_{[1,\infty)(k)}$ \\ 
 \hline
Derivative & $\partial_x$ & $2\pi ik$ \\
\hline
Absolute derivative &  $\Lambda$ & $2\pi |k|$\\
\hline
Mollifier & $Q^N$ & $\mathbf{1}_{[0,N]}(k)$ \\
\hline
\end{tabular}
\end{center}

\subsection*{Acknowledgements.} Adrian Constantin was supported by the Austrian Science Fund (FWF), grant number Z 387-N. Pierre Germain was supported by a Wolfson fellowship from the Royal Society and the Simons collaboration on Wave Turbulence.

\section{Structure of the equation}

\label{sectionstructure}

\subsection{The energy}
The energy associated to \eqref{eqC} is
$$
\mathcal{E}_\sigma(u) = \frac{1}{4\pi^2} \iint_0^{2\pi} \frac{|u(x) - u(y)|^4}{1-\cos (x-y)} \dd x \dd y - \frac{2 \sigma}{\pi} \int_0^{2\pi} |u(x)|^4 \dd x.
$$

The second summand is obviously the $L^4$ norm raised to the power $4$, while the first summand is equivalent to the norm of the fractional Sobolev space $W^{\frac 14,4} = B^{\frac 14}_{4,4}$ raised to the power 4 (it is the so-called Gagliardo norm). Namely, if $u$ has mean zero, denoting $P_j$ for the usual Littlewood-Paley projection,
$$
\iint_0^{2\pi} \frac{|u(x) - u(y)|^4}{1-\cos (x-y)} \dd x \dd y \sim \| u \|_{B^{\frac 14}_{4,4}}^4 = \sum_{j \geq 1} 2^j \| P_j f \|_{L^4}^4 ,
$$
see Triebel \cite{Triebel}, Section 2.6.1.

The equation \eqref{eqC} can then be interpreted as the symplectic flow associated with this energy through the symplectic form $\omega(u,v) = \operatorname{Re} \int_0^{2\pi}\partial_x^{-1} \overline{u} \,  v \dd x$, see Appendix \ref{symplecticappendix}.

\subsection{Form of the filamentation equation in Fourier space} 
By Lemma \ref{computation2}, the cubic term can be written
\begin{equation}
\label{newformulaC}
\mathcal{C}_\sigma[u] =  \mathbb{P} \left[ \frac{1}{2\pi} |u|^2 \Lambda u - \frac{1}{2\pi} u \Lambda |u|^2 - \sigma |u|^2 u \right] = \sum_{\substack{p = k+\ell - m \\ k,\ell,m,p \geq 1}} e^{ipx} \left( k - |k-p| - \sigma \right) a_k a_\ell \overline{a_m} ,
\end{equation}
where $\Lambda = |\partial_x|$ corresponds to the Fourier multiplier $2\pi |k|$.
Symmetrizing in $k$ and $\ell$ and using the remarkable identity 
$$
\frac{1}{2} (k + \ell - |k-p| - |\ell -p| ) = \min(k,\ell,m,p),
$$
already noted by \cite{BH}, we obtain yet another formulation
\begin{equation}
\label{newformulaC2}
\mathcal{C}_\sigma[u] =  \sum_{\substack{p = k+\ell - m \\ k,\ell,m,p \geq 1}} e^{ipx} \left( \min(k,\ell,m,p) - \sigma \right) a_k a_\ell \overline{a_m} .
\end{equation}

By Lemma \ref{computation3}, the energy can be rewritten using the Fourier multiplier $\Lambda$ as
\begin{equation}
\label{newformulaE}
\mathcal{E}_\sigma[u] = \frac{1}{2\pi} \int_0^{2\pi} \left[ - 4 |u|^2 \Lambda |u|^2 + 2 |u|^2 (\overline{u} \Lambda u + u \Lambda \overline{u}) \right] \dd z - \frac{2\sigma}{\pi} \int_0^{2\pi} |u(x)|^4 \dd x .
\end{equation}
This form is more amenable to the subsequent analysis.

\subsection{Symmetries and invariants}

The symmetries of the filamentation equation are:
\begin{itemize}
\item space translation $u \mapsto u(x+x_0)$, $x_0 \in \mathbb{T}$;
\item phase rotation $u \mapsto e^{i \theta} u$, $\theta \in \mathbb{R}$;
\item scaling $u(t,x) \mapsto \lambda u(\lambda^2 t,x)$, $\lambda \in \mathbb{R}$; and
\item space-time reversal $u(t,x) \mapsto u(-t,-x)$.
\end{itemize}

The following quantities are conserved: 
\begin{itemize}
\item if $\sigma=1$, the first Fourier mode $\displaystyle a_1 = \frac{1}{2\pi} \int_0^{2\pi} u(x) e^{-ix} \dd x$;
\item the energy $\mathcal{E}_\sigma(u)$ of course;
\item the `momentum' ($L^2$ norm) $\displaystyle \mathcal{P}(u) = \int |u(x)|^2 \dd x$; and
\item the `mass' ($\dot H^{-\frac 12}$ norm) $\displaystyle \mathcal{M}(u) = \int \left| \partial_x^{-\frac 12} u \right|^2 \dd x$ (only defined if $a_0=0$, which can be assumed to hold without loss of generality, see \cite{DCG})
\end{itemize}

Let us briefly justify these conservation laws. The conservation of $a_1$ if $\sigma=1$ follows immediately from the right-hand side of equation \eqref{newformulaC}, which vanishes if $p=\sigma=1$ and $k \geq 1$. The conservation of the mass and momentum are most easily obtained from \eqref{newformulaC2} and symmetrization:
\begin{align*}
& \frac{d}{dt} \| a_k \|_{\ell^2}^2 = 2 \operatorname{Re} i \sum_{\substack{p = k+\ell - m \\ k,\ell,m,p \geq 1}} p \left( \min(k,\ell,m,p) - \sigma \right) a_k a_\ell \overline{a_m a_p} \\
& \qquad \qquad = \frac 1 2 \operatorname{Re} i \sum_{\substack{p = k+\ell - m \\ k,\ell,m,p \geq 1}} p \left( \min(k,\ell,m,p) - \sigma \right) a_k a_\ell \overline{a_m a_p} = 0, \\
& \frac{d}{dt} \| a_k \|_{\ell^2}^2 = 2 \operatorname{Re} i \sum_{\substack{p = k+\ell - m \\ k,\ell,m,p \geq 1}} \left( \min(k,\ell,m,p) - \sigma \right) a_k a_\ell \overline{a_m a_p} = 0.
\end{align*}

As was already noticed in \cite{BH}, the conserved quantities and symmetries are related via Noether's theorem; see Appendix \ref{symplecticappendix} for details.

\section{Local strong solutions}

\label{sectionstrong}

\begin{theorem}
The filamentation equation is locally well-posed in $H^s$ if $s> \frac 32$. More precisely, if $u_0 \in H^s$, there exists $T>0$ and a unique solution $u \in \mathcal{C}([0,T],H^s)$ to the Cauchy problem
$$
\begin{cases}
& \partial_t u = \partial_x \mathcal{C}_\sigma[u] \\
& u(t=0) = u_0.
\end{cases}
$$
\end{theorem}

\begin{proof} We shall focus on the case $\sigma=0$ for simplicity, and only prove the \textit{a priori} estimate. Using the identity \eqref{newformulaC2} and symmetrizing in $k,\ell,m,p$, we find
\begin{align*}
\frac{d}{dt} \| u \|_{H^s}^2
& = (2\pi)^{2s+1} \operatorname{Re} i \sum_{k+\ell =m +p} p^{2s+1} \min(k,\ell,m,p) a_k a_\ell \overline{a_m a_p} \\
& = (2\pi)^{2s+1} \operatorname{Re} i \sum_{k+\ell =m +p} (p^{2s+1}+m^{2s+1}-k^{2s+1} - \ell^{2s+1}) \min(k,\ell,m,p) a_k a_\ell \overline{a_m a_p} .
\end{align*}
Without loss of generality, we can restrict the summation above to $\ell \leq k$ and $m \leq p$; as a result, $p \sim k$.
Then repeated applications of the Cauchy-Schwarz inequality combined with the hypothesis that $s>\frac 32$ give
$$
\left| \sum_{\substack{k+\ell =m +p \\ \ell \leq k, \, m \leq p}} \ell^{2s+1} \min(k,\ell,m,p) a_k a_\ell \overline{a_m a_p} \right| \leq
\sum_{\substack{k+\ell =m +p \\ \ell \leq k, \, m \leq p}} k^s p^s \ell k \, a_k a_\ell \overline{a_m a_p} \lesssim \| k^s a_k \|_{\ell^2}^2 \| k a_k \|_{\ell^1}^2
\lesssim \| k^s a_k \|_{\ell^2}^4.
$$
Symmetrically,
$$
\left| \sum_{\substack{k+\ell =m +p \\ \ell \leq k, \, m \leq p}} m^{2s+1} \min(k,\ell,m,p) a_k a_\ell \overline{a_m a_p} \right|  \lesssim \| k^s a_k \|_{\ell^2}^4.
$$

There remains to bound
$$
\left| \sum_{\substack{k+\ell =m +p \\ \ell \leq k, \, m \leq p}} (p^{2s+1} - k^{2s+1}) \min(k,\ell,m,p) a_k a_\ell \overline{a_m a_p} \right|.
$$
Using the fact that $(p^{2s+1} - k^{2s+1}) \min(k,\ell,m,p) \lesssim p^s k^s (\ell + m) \min(k,\ell,m,p)$, and proceeding as earlier, this can be bounded by
$$
\dots \lesssim \| k^s a_k \|_{\ell^2}^4.
$$
Overall, we have proved the \textit{a priori} estimate
$$
\frac{d}{dt} \| u \|_{H^s}^2 \lesssim \| u \|_{H^s}^4,
$$
which gives an \textit{a priori} local bound on $\| u \|_{H^s}$.
\end{proof}

\section{Global weak solutions}

\label{sectionweak}

For $u$ and $\varphi$ sufficiently regular and with positive spectrum, symmetrization gives the following identity
$$
\int \partial_x \mathcal{C}_\sigma (u) \overline{\varphi(x)} \d x
= - \frac{1}{8\pi} \iint \frac{|u(x) - u(y)|^2 (u(x) - u(y))\overline{(\partial_x \varphi(x) - \partial_x \varphi(y))} }{1 - \cos(x-y)} \dd x \dd y - \sigma \int |u(x)|^2 u(x) \varphi(x) \dd x.
$$
Notice that the right-hand side is well defined if $\varphi$ is smooth and $u \in W^{\frac 14 ,4}$; indeed,
$$
\left| \iint \frac{|u(x) - u(y)|^2 (u(x) - u(y))\overline{(\partial_x \varphi(x) - \partial_x \varphi(y))} }{1 - \cos(x-y)} \dd x \dd y \right| 
\lesssim \left\| \frac{\partial_x \varphi(x) - \partial_x \varphi(y)}{|x-y|} \right\|_{L^\infty_{x,y}} \left\| \frac{|u(x) - u(y)|^3}{|x-y|} \right\|_{L^1_{x,y}}.
$$

This leads to the following definition of a weak solution.

\begin{definition}
If $T \in (0,\infty]$, a function $u \in L^\infty_t([0,T], W^{\frac{1}{4},4})$ is a weak solution of \eqref{eqC} with data $u_0 \in W^{\frac{1}{4},4})$ if and only if
\begin{align*}
& \int_0^{2\pi} u_0 (x) \overline{\varphi(t=0,x)} \dd x + \int_0^T \int_0^{2\pi} u(t,x) \overline{\partial_t \varphi(t,x)} \dd x \dd t \\
& = - \frac{1}{8\pi} \int_0^T \iint_0^{2\pi} \frac{|u(x) - u(y)|^2 (u(x) - u(y))\overline{(\partial_x \varphi(x) - \partial_x \varphi(y))} }{1 - \cos(x-y)} \dd x \dd y \dd t - \sigma \int_0^T \int_0^{2\pi} |u(x)|^2 u(x) \overline{ \varphi(x)} \dd x \dd t.
\end{align*}
for any $\varphi(t,x)$ which is smooth, is compactly supported on $[0,T) \times \mathbb{T}$ and has a positive spectrum $\mathbb{P}_+ \varphi = \varphi$ (in the second line, we omitted the dependence on $t$ for sake of clarity).
\end{definition}

\begin{theorem}
If $u_0 \in W^{\frac{1}{4},4}(\mathbb{T})$, there exists a global weak solution $u \in L^\infty W^{\frac 14,4}$ of \eqref{eqC}.
\end{theorem}

\begin{proof} \underline{Step 1: regularization of the equation.} Define the regularized energy
$$
\mathcal{E}_\sigma^N(u) = \mathcal{E}_\sigma (Q^N u), \qquad \mbox{where $Q^N = \mathbf{1}_{[0,N]}$.}
$$
It generates the Hamiltonian flow
$$
\partial_t u = \partial_x Q^N \mathcal{C}_\sigma (Q^N u).
$$
The solution associated with the data $u_0$ will be denoted $u^N$. Observe that $(\operatorname{Id} - Q^N) u$ is constant, so that $u^N$ is effectively finite-dimensional, and therefore global in time by virtue of the global bound provided by conservation of the energy.

\medskip

\noindent \underline{Step 2: compactness.} By energy conservation, $u^N$ is uniformly bounded in $L^\infty ([0,\infty), W^{\frac 14,4}(\mathbb{T}))$. Furthermore, we claim that
$$
\| \mathcal{C}_\sigma (u) \|_{H^{-1}} \lesssim \| u \|_{W^{\frac 14,4}}^3.
$$
By symmetrization, H\"older's inequality, the Gagliardo characterization of $W^{\frac 14,4}$, and Sobolev embedding, we have
\begin{align*}
\left| \int \mathcal{C}_0 (u)(x) \overline{g(x)} \dd x \right| & = \left| \frac{1}{8\pi} \iint \frac{|u(x) - u(y)|^2 (u(x) - u(y))\overline{(g(x) - g(y))} }{1 - \cos(x-y)} \dd x \dd y \right| \\
& \lesssim \left\| \frac{|u(x) - u(y)|^3}{|x-y|^{\frac 32}} \right\|_{L^{\frac 43}} \left\| \frac{g(x) - g(y)}{|x-y|^{\frac 12}} \right\|_{L^4} \lesssim  \| u \|_{W^{\frac 14,4} }^3 \| g \|_{W^{\frac 14,4} } \lesssim  \| u \|_{W^{\frac 14,4} }^3 \| g \|_{H^{\frac 12}}.
\end{align*}
Here, $H^{\frac 12}$ is the classical $L^2$-based Sobolev space, which we fall back to in order to avoid technicalities linked to fractional Sobolev spaces with $p \neq 2$.

Interpreting the above inequality by duality and using $L^p$ boundedness of the Fourier multiplier $Q^N$, we have
$$
\| \partial_x Q^N  \mathcal{C}_0(Q^N u) \|_{H^{-{\frac 32}}} \lesssim \| \mathcal{C}_0(Q^N u) \|_{H^{-{\frac 12}}} \lesssim \| Q^N u \|_{W^{\frac 14,4} }^3.
$$
While we have focused on $\mathcal{C}_0$ for notational simplicity, the bound above remains true for $\mathcal{C}_1$.

By the equation satisfied by $u^N$, this gives a uniform bound for $\partial_t u^N$ in $L^\infty H^{-{\frac 32}}$; recall furthermore the uniform bound in $L^\infty W^{\frac 14,4}$. By the Rellich-Kondrakov theorem, $W^{\frac 14,4}$ embeds compactly in $L^{10}$, and by the Sobolev embedding theorem, $L^{10}$ embeds continuously in $H^{-\frac 32}$. Thus, the Aubin-Lions lemma gives compactness of $(u^N)$ in $L^\infty L^{10}$. We select a subsequence converging in $L^\infty L^N$, and still denote it $(u^N)$ for simplicity.

Finally, the sequence $(u^N)$ is uniformly bounded in $L^\infty_t W^{\frac 14,4}$. Therefore, by weak compactness of this space, we can also assume (up to extracting a new subsequence) that $(u^N)$ converges to $u$ weakly in $L^\infty_t W^{\frac 14,4}$. This follows by reflexivity of $W^{\frac 14,4}$, see \cite{Lemarie,Triebel}.

\medskip

\noindent \underline{Step 3: passing to the limit in the equation.} The weak form of the equation satisfied by $u^N$ is
\begin{align*}
& \int_0^{2\pi} u_0 (x) \overline{\varphi(t=0,x)} \dd x + \int_0^\infty \int_0^{2\pi} u^N(t,x) \overline{\partial_t \varphi(t,x)} \dd x \dd t \\
& \quad = - \frac{1}{8\pi} \int_0^\infty \iint_0^{2\pi} \frac{|Q^N u^N(x) - Q^N u^N(y)|^2 (Q^N u^N(x) - Q^N u^N(y))\overline{(Q^N \partial_x \varphi(x) - Q^N \partial_x \varphi(y))} }{1 - \cos(x-y)} \dd x \dd y \dd t.
\end{align*}
(Once again, we have focused on the case $\sigma=0$ for the sake of clarity.) Passing to the limit in the first line is immediate, so that we turn to the second line, which we split into two pieces:
\begin{align*}
& \int_0^\infty \iint_0^{2\pi} \frac{|Q^N u^N(x) - Q^N u^N(y)|^2 (Q^N u^N(x) - Q^N u^N(y))\overline{(Q^N \partial_x \varphi(x) - Q^N \partial_x \varphi(y))} }{1 - \cos(x-y)} \dd x \dd y \dd t \\
& \qquad = \int_0^T \iint_{S_\delta} \dots  \dd x \dd y \dd t + \int_0^T \iint_{\mathbb{T} \setminus S_\delta} \dots  \dd x \dd y \dd t,
\end{align*}
where $S_\delta$ is the $\delta$-neighbourhood of the diagonal: $S_\delta = \{ (x,y) \in \mathbb{T}^2, \, |x-y| < \delta \}$. The first piece goes to zero with $\delta$, as follows by H\"older's inequality:
\begin{align*}
| I_\delta | \leq \left\| \frac{|Q^N u^N(x) - Q^N u^N(y)|^3}{|x-y|^{\frac 32}} \right\|_{L^{\frac 43}} \left\| \frac{Q^N \partial_x \varphi(x) - Q^N \partial_x \varphi(y)}{|x-y|^{\frac 12}} \right\|_{L^8} \| \mathbf{1}_{S_\delta} \|_{L^8} \lesssim \| u^N \|_{W^{\frac 14,4}} \| \partial_x \varphi \|_{W^{\frac 38,8}} \delta^{\frac 18}.
\end{align*}
By compactness in $L^\infty L^{10}$ and convergence of $Q^N \partial_x \varphi$ to $\partial_x \varphi$ in $L^p$, the second piece has the desired limit as $N \to \infty$:
\begin{align*}
II_\delta \overset{N \to \infty}{\longrightarrow} \int_0^\infty \iint_0^{2\pi} \frac{|u(x) - u(y)|^2 (u(x) - u(y))\overline{( \partial_x \varphi(x) -  \partial_x \varphi(y))} }{1 - \cos(x-y)} \dd x \dd y \dd t.
\end{align*}
Combining the bound on $I_\delta$ and the limit of $II_\delta$ gives the desired result.
\end{proof}

\section{Variational aspects}

\label{sectionvariational}

\subsection{Sign of the energy}

\begin{lemma}
\label{lemmanonnegative}
\begin{itemize}
\item[(i)] If $u$ has a non-negative spectrum ($\mathbb{P}u=u$), then $\mathcal{E}_0(u) \geq 0$, and it is only zero if $u =C$, $C\in \mathbb{C}$.
\item[(ii)] If $u$ has a positive spectrum ($\mathbb{P}_+u=u$), then $\mathcal{E}_1(u) \geq 0$, and it is only zero if $u = Ce^{ix}$, $C\in \mathbb{C}$.
\end{itemize}
\end{lemma}

\begin{proof}
The first assertion is obvious. We learn from 
\eqref{newformulaE} that
$$
\mbox{if} \; \mathbb{P}_+ u =u, \qquad \mathcal{E}_0(e^{-ix}u) = \mathcal{E}_1(u),
$$
which implies the second assertion.
\end{proof}

\subsection{Weak compactness and lower semi-continuity of the energy}

We provide a direct proof of weak compactness and lower semi-continuity for the energy which avoids any functional analytic prerequisite, except for the characterization $\mathcal{E}_0(u) \sim \| u \|_{W^{\frac 14,4}}$. 

\begin{theorem}
\label{LSC}
Assume that $(u_n)$ is a sequence in $W^{\frac 14,4}$ such that $\mathcal{E}_0(u_n)$ is decreasing. Then there exists a subsequence (which we still denote $(u_n)$ for simplicity) and a function $u \in W^{\frac 14,4}$ such that, for any $\varphi \in \mathcal{C}^\infty$
$$
\int u_n(x) \varphi(x) \dd x \overset{n \to \infty}{\longrightarrow} \int u(x) \varphi(x) \dd x ,
$$
and furthermore
$$
\mathcal{E}_0(u) \leq \lim_{n \to \infty} \mathcal{E}_0(u_n).
$$
\end{theorem}

\begin{proof} \underline{Step 1: weak compactness.} Recall that $Q^N$ is the Fourier multiplier with symbol $\mathbf{1}_{[1,N]}(k)$. For any $N \in \mathbb{N}$, the sequence $u^N_n = Q^N u_n$ is compact in $L^4$; hence, by a diagonal argument, we can find Fourier coefficients $(a_k)$  and a subsequence (in $n$) such that $u^N_n \rightarrow \sum_{k=1}^N a_k e^{ikx}$ as $n \to \infty$ for any $N$. This implies that $\sum_{k=1}^N a_k e^{ikx}$ is uniformly (in $N$) bounded in $W^{\frac 14,4}$. From this, it is not hard to deduce that $u = \sum_{k=1^\infty} a_k e^{ikx}$ belongs to $W^{\frac{1}{4},4}$. 

Splitting $u_n$ into $Q^N u_n$ and $(1-Q^N) u_n$, one readily establishes that
$$
\int u_n \varphi(x) \dd x \overset{n \to \infty}{\longrightarrow} \int u(x) \varphi(x) \dd x.
$$

\medskip

\noindent \underline{Step 2: lower semi-continuity.} We will now split $u_n$ into a low, a middle, and a high frequency part
\begin{align*}
u_n = Q^N u_n + (Q^M - Q^N) u_n + (1-Q^M) u_n = u_n^{\operatorname{low}} + u_n^{\operatorname{mid}} +  u_n^{\operatorname{high}},
\end{align*}
with a similar decomposition for $u$. This decomposition depends on two parameters, 
N and M, 
which we fix as follows. First, we suppose we have an $\epsilon>0$, and $N$ such that
\begin{equation}
\label{choiceN}
\| u - u^{\operatorname{low}}\|_{W^{\frac 14,4}} < \epsilon,
\end{equation}
and finally $M$ such that 
\begin{equation}
\label{choiceM}
\iint \frac{|u^{\operatorname{low}}(x) - u^{\operatorname{low}}(y)|^2}{1 - \cos(x-y)} \operatorname{Re} \left[(\overline{u^{\operatorname{low}}(x) - u^{\operatorname{low}}(y) })(f^{\operatorname{high}}(x) - f^{\operatorname{high}}(y))\right] \dd x \dd y < \epsilon \qquad \mbox{if $\| f \|_{W^{\frac 14,4}} < 1$}
\end{equation}
(the existence of such an $M$ is a consequence of the fact that $\frac{|u^{\operatorname{low}}(x) - u^{\operatorname{low}}(y)|^2 (\overline{u^{\operatorname{low}}(x) - u^{\operatorname{low}}(y) })}{1 - \cos(x-y)} $ is a given smooth function after $N$ has been fixed).

Once $N$ and $M$ have been fixed, we choose a constant $K$ such that, if $n>K$, then 
\begin{equation}
\label{coccinelle}
\begin{split}
& \| u^{\operatorname{low}}_n - u^{\operatorname{low}} \|_{W^{\frac 14,4}} < \epsilon \\
& \| u^{\operatorname{mid}}_n \|_{W^{\frac 14,4}} \lesssim \epsilon
\end{split}
\end{equation}
(such a $K$ exists since $u^{\operatorname{low}}_n$ and $u^{\operatorname{mid}}_n$ converge to $u^{\operatorname{low}}$ and $u^{\operatorname{mid}}$ respectively, as $n \to \infty$).

Now that all the parameters are fixed, we proceed with the inequalities. By H\"older's inequality and \eqref{coccinelle}, if $n>K$,
\begin{align*}
\mathcal{E}_0(u_n) =  \mathcal{E}_0(u_n^{\operatorname{low}} - u^{\operatorname{low}} +u^{\operatorname{low}} + u_n^{\operatorname{mid}} +  u_n^{\operatorname{high}})
=  \mathcal{E}_0(u^{\operatorname{low}} + u_n^{\operatorname{high}}) + O(\epsilon).
\end{align*}
By convexity, for any $w,z \in \mathbb{C}$, $|w|^4 \geq |z|^4 + 4 |z|^2 \operatorname{Re} \overline{z}(w-z)$. Combining this inequality with \eqref{choiceM}, we have
\begin{align*}
\mathcal{E}_0(u^{\operatorname{low}} + u_n^{\operatorname{high}}) & \geq \mathcal{E}_0(u^{\operatorname{low}}) + \iint \frac{|u^{\operatorname{low}}(x) - u^{\operatorname{low}}(y)|^2}{1 - \cos(x-y)} \operatorname{Re} \left[(\overline{u^{\operatorname{low}}(x) - u^{\operatorname{low}}(y) })(u_n^{\operatorname{high}}(x) - u_n^{\operatorname{high}}(y))\right] \dd x \dd y \\
& = \mathcal{E}_0(u^{\operatorname{low}}) + O(\epsilon).
\end{align*}
Finally, it follows from \eqref{choiceN} that
$$
\mathcal{E}_0(u^{\operatorname{low}}) = \mathcal{E}_0(u) + O(\epsilon).
$$
Overall, we proved that, for any $\epsilon>0$, there exists $K$ such that, if $n>K$,
$$
\mathcal{E}_0(u_n) \geq \mathcal{E}_0(u) + O(\epsilon).
$$
This gives the desired result!
\end{proof}

\subsection{Constrained minimization}

\begin{proposition} \label{constrainedmin}
If $\mathcal{M}^*, \mathcal{P}^* > 0$, let
$$
\mathcal{E}^*_\sigma = \mathcal{E}_\sigma(\mathcal{M}^*, \mathcal{P}^*) = \inf \{ \mathcal{E}_\sigma(u), \; u \in W^{\frac 14,4}, \; \mathbb{P}_+ u = u, \; \mathcal{M}(u) = \mathcal{M}^*, \; \mathcal{P}(u) = \mathcal{P}^* \} .
$$
Then this infimum is reached at a minimizer.
The same holds true if only the momentum or only the mass is fixed.
\end{proposition}

\begin{proof} 
We apply the direct method of the calculus of variations and consider a minimizing sequence $(u_n)$ such that
\begin{itemize}
\item $u_n \in W^{\frac 14,4}$ for any $n$,
\item $\mathbb{P}_+ u_n = u_n$ for any $n$,
\item $\mathcal{M}(u_n) = \mathcal{M}^*$ for any $n$,
\item $\mathcal{P}(u_n) = \mathcal{P}^*$ for any $n$, and
\item $\mathcal{E}_\sigma(u_n) \rightarrow \mathcal{E}_\sigma^*$ as $n \to \infty$.
\end{itemize}

Since $(u_n)$ is bounded in $W^{\frac 14,4}$, it is precompact in $L^p$, $p<\infty$ by the Rellich-Kondrakov theorem and there exists a subsequence (still denoted $(u_n)$) such that $u_n \rightarrow u$ in $L^{10}$ and in particular $\mathcal{M}(u) = \mathcal{M}^*$ and $\mathcal{P}(u) = \mathcal{P}^*$. We now apply Theorem \ref{LSC} and obtain yet another subsequence, still denoted $(u_n)$, and a limit $\widetilde{u} \in W^{\frac 14,4}$; however, by uniqueness of the (weak) limit, it follows that $\widetilde{u} = u$.

Overall, we have found $u \in W^{\frac 14,4}$ such that $\mathcal{M}(u) = \mathcal{M}^*$, $\mathcal{P}(u) = \mathcal{P}^*$, and $\mathcal{E}_\sigma(u) \leq \mathcal{E}_\sigma^*$. By definition of $\mathcal{E}_\sigma^*$, actually $\mathcal{E}_\sigma(u) \leq \mathcal{E}_\sigma^*$ holds, which is the desired statement.
\end{proof}

The Euler-Lagrange equation satisfied by the minimizer obtained in the previous proposition is
$$
\partial \mathcal{E}_\sigma(u) = \lambda \partial \mathcal{M}(u) + \mu \partial \mathcal{P}(u),
$$
where the Lagrange multipliers $\lambda,\mu \in \mathbb{C}$.
Using the formula in \eqref{gradients} for the gradients of $\mathcal{E}_\sigma$, $\mathcal{M}$ and $\mathcal{P}$, the Euler-Lagrange equation becomes
$$
\frac 4 \pi \mathcal{C}_\sigma(u) = \lambda \Lambda^{-1} u + \mu u.
$$
This equation is satisfied \textit{a priori} in a weak sense, namely after pairing with a smooth function. It would be interesting to develop a regularity theory; i.e.\ is a solution of this equation or a minimizer necessarily smooth? The equation would then be satisfied in a stronger sense.

\section{Traveling waves}

\label{sectiontraveling}

\subsection{General traveling waves} These have the form
$$
\Phi(x-ct) e^{i \omega t},
$$
where $\Phi$ has positive spectrum. Plugging this ansatz in \eqref{eqC} gives the equation
$$
- c \partial_x \Phi + \omega \Phi = \partial_x \mathcal{C}_\sigma[\Phi]
$$
or after integrating
\begin{equation}
\label{eqPhi}
- c \Phi + \omega \Lambda^{-1} \Phi = \mathcal{C}_\sigma[\Phi]
\end{equation}
(in which sense this equation is satisfied depends on the smoothness of $\Phi$; if $\Phi$ is only an element of $W^{\frac 14,4}$, then this equation is interpreted in a weak sense).

Taking the inner product with $\overline{\Phi}$ and using \eqref{gradientsandfunctions}, we find the identity
$$
- c \mathcal{P}(\Phi) + \omega \mathcal{M}(\Phi) = \frac{\pi}2 \mathcal{E}_\sigma(\Phi).
$$

\begin{lemma}
By definition, stationary solutions are constant in time (i.e.\ satisfy the above with $c=\omega=0$).
\begin{itemize}
\item[(i)] If $\sigma=0$, the only stationary solution is zero.
\item[(ii)] If $\sigma =1$, the only stationary solutions are multiples of $e^{ix}$.
\end{itemize}
\end{lemma}

\begin{proof} This follows from the above identity and Lemma \ref{lemmanonnegative}. \end{proof}

\subsection{Explicit formulas}

We know that exact traveling waves are given by 
$$
\Psi_k(t,x) = e^{i[ kx + k(k - \sigma) t]}, \qquad k \in \mathbb{N}.
$$
Indeed, if $k \geq 0$, 
$$
\mathcal{C}_\sigma[e^{ikx}] = (k-\sigma) e^{ikx}.
$$
These traveling waves are such that
$$
\mathcal{P}(\Psi_k) = 2\pi, \qquad \mathcal{M}(\Psi_k) = k^{-1}, \qquad \mathcal{E}_\sigma(\Psi_k) = 4(k-\sigma).
$$

\begin{proposition}
\begin{itemize}
\item[(i)] If $\sigma=0$, the traveling wave $\Psi_1$ is orbitally stable.
\item[(ii)] If $\sigma =1$, the traveling waves $\Psi_1$ and $\Psi_2$ are orbitally stable.
\end{itemize}
\end{proposition}

\begin{proof}
This follows from the conservation of $\mathcal{M}$ and $\mathcal{P}$ and from the conservation of $a_1$ if $\sigma=1$.
\end{proof}

If $\sigma=1$, adding the first Fourier mode to these traveling waves gives an Ansatz which is conserved. The corresponding explicit solutions are
$$
A e^{ix} + B e^{i\left[kx + k(k-1)(2 |A|^2 + |B|^2)t \right]}, \qquad A,B \in \mathbb{C}, \; k \in \mathbb{N}.
$$ 
These solutions are reminiscent of Rossby-Haurwitz waves, which play an important role in planetary fluid flows \cite{Haurwitz1940,Vallis2006}.

\subsection{Variational construction}
Proposition \eqref{constrainedmin} provides solutions of the constrained minimization problem
$$
\min \mathcal{E}_\sigma(\Phi) \quad \mbox{subject to} \quad \mathcal{M}(\Phi) = \mathcal{M}^* \;\; \mbox{and} \;\; \mathcal{P}(\Phi) = \mathcal{P}^*.
$$
The corresponding Euler-Lagrange equation
$$
\frac{4}{\pi} \mathcal{C}_\sigma[\Phi] = \lambda \Lambda^{-1} \Phi  + \mu \Phi.
$$
(where $\lambda$ and $\mu$ are Lagrange multipliers) coincides with \eqref{eqPhi}. This construction gives traveling wave solutions for any fixed $\mathcal{M}^*$ and $\mathcal{P}^*$

\appendix

\section{Variational and symplectic structures}

\label{appendixA}

\label{symplecticappendix}

We formally present the variational and symplectic structures of the filamentation equation. This will give a link between the symmetries and the conserved quantities that were presented in the introduction.

\subsection{General symplectic formalism}

We follow the presentation in Tao \cite{Tao} and start with a symplectic (antisymmetric, non-degenerate) bilinear form $\omega$ on a vector space $E$.

Given a real function $F$ on $E$, its differential is such that
$$
F(u+\epsilon \varphi) = F(u) + \epsilon \, dF(u) \cdot \varphi + O(\epsilon^2).
$$

By definition, the symplectic gradient $\nabla_\omega F$ is such that
$$
dF(u) \cdot \varphi = \omega( \nabla_\omega F , \varphi).
$$

The symplectic flow associated to $F$ is the evolution equation on $X$ given by
$$
\dot u = \nabla_\omega F(u).
$$
Given data $u_0$, we denote the solution $u$ to this evolution equation by
$$
u(t) = \Phi^\omega_F(u_0,t).
$$

We consider the flow given by a specific Hamiltonian $\mathcal{H}$
$$
\dot u = \nabla_\omega \mathcal{H} (u).
$$

Then the following are equivalent (Noether's theorem):
\begin{itemize}
\item the Poisson bracket vanishes: $\{ \mathcal{H}, F \} = \omega(\nabla_\omega \mathcal{H}, \nabla_\omega F) =0 $;
\item $F$ is a conserved quantity for the flow of $\mathcal{H}$: $F( \Phi^\omega_\mathcal{H}(u_0,t))$ is independent of $t$; and
\item $\Phi^\omega_F(\cdot,t)$ is a symmetry of the Hamiltonian $\mathcal{H}$; therefore, $\mathcal{H}( \Phi^\omega_{F}(u_0,t))$ is independent of $t$.
\end{itemize}

\subsection{Gradients}

Given a real functional $F$ om complex functions $u$ defined on $\mathbb{T}$ with positive spectrum, we denote $\partial F(u)$ as the gradient which is such that
$$
F(u+\epsilon \varphi) = F(u) + 2 \epsilon \operatorname{Re} \int \partial F(u) \varphi \dd x + O(\epsilon^2)
$$
(the usual notation for this function is $\partial F(u) = \frac{\partial F(u)}{\partial u}$, where the derivative is understood in the complex sense). The gradient $\partial F(u)$ becomes unique once we require that it has a negative spectrum.

The gradients of the Hamiltonian, momentum and mass are given by
\begin{equation}
\label{gradients}
\begin{split}
& \partial \mathcal{E}_\sigma(u) = \frac{4}{\pi} \overline{\mathcal{C}(u)} \\
& \partial \mathcal{P}(u) = \overline{u} \\
& \partial \mathcal{M}(u) = \Lambda^{-1} \overline{u}.
\end{split}
\end{equation}

Furthermore, the functions and their gradients are related through the inner products
\begin{equation}
\label{gradientsandfunctions}
\begin{split}
& \int \overline{\mathcal{C}(u)} \, u \dd x = \frac{\pi}2 {\mathcal{E}_\sigma(u)} \\
& \int \overline{\partial \mathcal{P}(u)} \, u \dd x = \mathcal{P}(u) \\
& \int \overline{\partial \mathcal{M}(u)} \, u \dd x= \mathcal{M}(u).
\end{split}
\end{equation}

\subsection{Application to the filamentation equation}

Consider the symplectic form on functions with positive spectrum
$$
\omega(u,v) = \operatorname{Re} \int_0^{2\pi}\partial_x^{-1} \overline{u} \,  v \dd x
$$

If the function $F$ admits a gradient as in the previous subsection, then
$$
\nabla_\omega F = 2 \partial_x\overline{ \partial F}
$$
(note that this formula makes sense even without the positive spectrum requirement, while the symplectic form does require a positive spectrum).

\begin{itemize}
\item For the energy $\mathcal{E}_\sigma$, the gradient is $\frac{4}{\pi} \overline{\mathcal{C}(u)}$ and the symplectic gradient is $\frac{8}{\pi} \partial_x \mathcal{C}[u]$. The associated symplectic flow is the equation \eqref{eqC}, up to time rescaling.
\item Consider the momentum $\mathcal{P}(u) = \int |u(x)|^2 \dd x$. Its gradient is simply $\partial \mathcal{P} = \overline u$, and its symplectic gradient is $\nabla_\omega \mathcal{P} = 2 \partial_x u$. This is the generator of the translations $u(x) \mapsto u(x+x_0)$.
\item Consider the mass $\mathcal{M}(u) = \int |\partial_x^{-\frac{1}{2}} u(x)|^2 \dd x$. Its gradient is $\partial \mathcal{M} = \Lambda^{-1} \overline u$, and its symplectic gradient is $\nabla_\omega \mathcal{M} = 2iu$. This is the generator of the phase rotations $u(x) \mapsto e^{i \theta} u(x)$.
\end{itemize}

\section{Computations}

\label{appendixB}

\begin{lemma} If $m \in \mathbb{Z}$, then
\begin{align*}
\int_0^{2\pi} \frac{1 - \cos(mz)}{1 - \cos z} \dd z = 2 \pi |m| \qquad \mbox{and} \qquad \lim_{\epsilon \to 0} \int_{\epsilon}^{2\pi - \epsilon}  \frac{1 - e^{i mz}}{1 - \cos z} \dd z = 2 \pi |m| .
\end{align*}
\end{lemma}
\begin{proof} The second identity is a simple consequence of the first; furthermore, it suffices to deal with the case $m >0$, which we will focus on. The integral can be written
$$
\int_0^{2\pi} \frac{1 - \cos(mz)}{1 - \cos z} \dd z = \int_S \frac{2 - z^m - z^{-m}}{2 - z - z^{-1}} \frac{\dd z}{iz} ,
$$
where the contour $S$ is the unit circle. The only pole of the integrated function is at $0$; the expansion of this function at the origin can be found by writing
$$
\frac{2 - z^m - z^{-m}}{2z - z^2 - 1}  = (z^m + z^{-m} -2 )(1 + z + z^2 + z^3 + \dots)^2 = (z^m + z^{-m} -2 )(1 + 2z + 3z^2 + 4z^3 + \dots),
$$
from which we see that the residue at $0$ is $m$. The result therefore follows by the residue theorem.
\end{proof}

\begin{lemma} If $u = \mathbb{P} u$, and $u$ is sufficiently smooth, then
\label{computation2}
\begin{align*}
\mathbb{P} \left[ \frac{1}{4\pi} \int_0^{2\pi} \frac{|u(x) - u(y)|^2 (u(x) - u(y))}{1-\cos (x-y)} \dd y \right] & = \frac{1}{2\pi} \mathbb{P} \left[ |u|^2 \Lambda u - u \Lambda |u|^2 \right] \\
& = \sum_{\substack{p = k+\ell - m \\ k,\ell,m,p \geq 0}} e^{ipx} \left( k - |k-p| \right) a_k a_\ell \overline{a_m} .
\end{align*}
\end{lemma}

\begin{proof}
Expanding $u$ in a Fourier series, $u(x) = \sum_{n \geq 0} a_n e^{inx}$, the left-hand side above becomes
\begin{align*}
& \sum_{\substack{k,\ell,m \geq 0\\ k+ \ell - m \geq 0}} a_k a_\ell \overline{a_m} \int_0^{2\pi} \frac{(e^{ikx} - e^{i ky})(e^{i\ell x} - e^{i \ell y})(e^{-imx} - e^{-i my})}{1- \cos(x-y)} \dd y \\
& \qquad = \sum_{\substack{k,\ell,m \geq 0 \\ k+ \ell - m \geq 0}}  a_k a_\ell \overline{a_m}  e^{i (k+\ell -m)x} \int_0^{2\pi} \frac{(1 - e^{i kz})(1 - e^{i \ell z})(1 - e^{-i mz})}{1- \cos z} \dd z.
\end{align*}
It remains to compute the integral. This is achieved by applying the previous lemma as follows:
\begin{align*}
&  \int_0^{2\pi} \frac{(1 - e^{i kz})(1 - e^{i \ell z})(1 - e^{-i mz})}{1- \cos z} \dd z \\
& \qquad =  \int_0^{2\pi} \frac{1 - e^{ikz} - e^{i\ell z} - e^{-imz} + e^{i(k+\ell)z} + e^{i(k-m)z} + e^{i(\ell -m) z} - e^{i(k+\ell -m)z}}{1 - \cos z} \dd z \\
& \qquad = \lim_{\epsilon \to 0} \int_\epsilon^{2\pi-\epsilon} \left[ \frac{1 - e^{i(k+\ell -m)z}}{1 - \cos z} + \frac{1 - e^{ikz}}{1- \cos z} + \frac{1 - e^{i\ell z}}{1- \cos z}  + \frac{1 - e^{im z}}{1- \cos z} \right. \\
& \qquad\qquad \qquad \qquad \qquad \qquad \qquad  \left. + \frac{e^{i(k+\ell)z} - 1 }{1 - \cos z} + \frac{e^{i(\ell-m)z} - 1 }{1 - \cos z} + \frac{e^{i(k-m)z} - 1 }{1 - \cos z}  \right] \dd z \\
& \qquad= 2\pi \left( k+\ell-m + k + \ell + m - (k+\ell) - |\ell-m| - |k-m| \right)\\
& \qquad= 2\pi \left( k + \ell - |\ell-m| - |k-m| \right)
\end{align*}
under the assumption that $k,\ell,m,k+\ell-m \geq 0$. Denoting $p=k+\ell-m$, we have by symmetry between $k$ and $\ell$,
$$
2\pi \sum \left( k + \ell - |\ell-m| - |k-m| \right) a_k a_\ell \overline{a_m} e^{ipx} = 4\pi \sum \left( k - |k-p| \right) a_k a_\ell \overline{a_m} e^{ipx}.
$$
Coming back to physical space, this is
$$
\mathbb{P} \int_0^{2\pi} \frac{|u(x) - u(y)|^2 (u(x) - u(y))}{1-\cos (x-y)} \dd y = 2 \mathbb{P} \left[ |u|^2 \Lambda u - u \Lambda |u|^2 \right] .
$$
\end{proof}

\begin{lemma}
\label{computation3} If $u = \mathbb{P} u$ and $u$ is sufficiently smooth,
$$
\iint_0^{2\pi} \frac{|u(x) - u(y)|^4 }{1-\cos (x-y)} \dd x \dd y =  2\pi \int_0^{2\pi} \left[ - 4 |u|^2 \Lambda |u|^2 + 2 |u|^2 (\overline{u} \Lambda u + u \Lambda \overline{u} \right] \dd z.
$$
\end{lemma}

\begin{proof}
The proof is similar to that of the previous lemma. After expanding $u$ in a Fourier series, we find that
$$
\iint_0^{2\pi}  \frac{|u(x) - u(y)|^4 }{1-\cos (x-y)} \dd x \dd y
= 2\pi \sum_{\substack{k,\ell,m,n >0 \\ k+\ell = m + n }} a_k a_\ell \overline{a_m a_n} \int_0^{2\pi} \frac{(1 - e^{i kz})(1 - e^{i \ell z})(1 - e^{-i mz})(1 - e^{-i nz})}{1- \cos z} \dd z.
$$
The integral can then be evaluated as
\begin{align*}
 &\int_0^{2\pi} \frac{(1 - e^{i kz})(1 - e^{i \ell z})(1 - e^{-i mz})(1 - e^{- i nz})}{1- \cos z} \dd z \\
 & \qquad  =  2\pi \left[ k + \ell + m + n - (k+\ell) - |k-m| - |k-n| - |\ell -m| - |\ell-n| - (m+n) + k + \ell + m + n \right] \\
 & \qquad = 2\pi \left[ k + \ell + m + n  - |k-m| - |k-n| - |\ell -m|  - |\ell-n| \right],
\end{align*}
where we used the facts that $k,\ell, m ,n \geq 0$ and $k + \ell = m+n$. This gives the desired result.
\end{proof}

\bibliographystyle{abbrv}
\bibliography{references}

\end{document}